\documentclass[11pt]{amsart}
\usepackage{amsmath,amssymb,amsthm}
\usepackage{booktabs,array}
\usepackage{microtype}
\usepackage{fancyvrb}
\usepackage{cite}
\usepackage[hidelinks]{hyperref}
\hypersetup{pdftitle={},pdfauthor={},pdfsubject={},pdfkeywords={},pdfcreator={},pdfproducer={}}
\ifdefined\pdfinfoomitdate\pdfinfoomitdate=1\fi
\ifdefined\pdftrailerid\pdftrailerid{}\fi
\ifdefined\pdfsuppressptexinfo\pdfsuppressptexinfo=15\fi

\newtheorem{theorem}{Theorem}[section]
\newtheorem{lemma}[theorem]{Lemma}
\newtheorem{proposition}[theorem]{Proposition}
\newtheorem{corollary}[theorem]{Corollary}
\theoremstyle{definition}
\newtheorem{definition}[theorem]{Definition}
\theoremstyle{remark}
\newtheorem{remark}[theorem]{Remark}
\numberwithin{equation}{section}
\numberwithin{table}{section}

\newcommand{\F}{\mathbb{F}}
\newcommand{\Z}{\mathbb{Z}}
\newcommand{\one}{\mathbf{1}}
\newcommand{\wt}{\operatorname{wt}}
\newcommand{\rank}{\operatorname{rank}}

\newcommand{\GL}{\operatorname{GL}}
\newcommand{\diag}{\operatorname{diag}}
\newcommand{\T}{\mathsf{T}}

\begin{document}

\title[Fuglede's conjecture in $\mathbb{F}_2^7$]{Fuglede's conjecture holds in $\mathbb{F}_2^7$}
\author{Atsushi Kadowaki}
\address{Knowledge Sense}
\email{math@knowledgesense.jp}
\date{6 October 2026}

\begin{abstract}
We prove that a nonempty subset of the binary vector space $\F_2^7$ is spectral if and only if it tiles $\F_2^7$ by translations. This is the smallest binary dimension left open by Ferguson and Sothanaphan. For the implication from tiling to spectrality we show that every eight-point subset of $\F_2^d$ admits $2^{d-3}$ frequencies whose pairwise differences avoid the zeros of its Fourier transform; consequently every tile of $\F_2^d$ with at least $2^{d-3}$ elements is spectral. For the converse implication, we use the Smith normal form of the Walsh matrix to exclude every spectral subset of $\F_2^7$, and of $\F_2^8$, whose size is not a power of two. The sizes $16$ and $32$ are reduced, by a projection argument, to the tiling of point sets attached to binary log-Hadamard matrices of small rank. Using the published classifications of Hadamard matrices of orders $16$ and $32$, we find that exactly sixteen order-$32$ classes up to transposition have dephased binary rank at most seven, and we give explicit tilings for both orientations of each.
\end{abstract}


\maketitle

\section{Introduction}\label{sec:intro}

Let $G$ be a finite abelian group. A nonempty set $E\subseteq G$ is \emph{spectral} if the space $L^2(E)$ of complex functions on $E$ has an orthogonal basis consisting of restrictions of characters of $G$. It \emph{tiles} $G$ if there is a set $T\subseteq G$ such that every element of $G$ has exactly one representation $e+t$ with $e\in E$ and $t\in T$; we then write $E\oplus T=G$. Fuglede's conjecture \cite{Fug74}, formulated for domains in $\mathbb{R}^d$, asserts that spectral sets and translational tiles coincide, and it makes sense in every locally compact abelian group. In Euclidean space it fails in both directions in dimension three and higher; the counterexamples were obtained from constructions in finite groups \cite{Tao04,KM06tiles,KM06hadamard,Mat05,FMM06}. In finite abelian groups both properties are finite combinatorial conditions.

For vector spaces over prime fields the conjecture holds in dimension at most two \cite{IMP17}. Aten et al.\ \cite{AAB17} developed a systematic theory for $\F_p^d$, including a correspondence between spectral pairs and log-Hadamard matrices of bounded rank. Ferguson and Sothanaphan \cite{FS20} disproved the conjecture in $\F_p^4$ for every odd prime $p$, proved it in $\F_2^4$, verified it by computer in $\F_2^5$ and $\F_2^6$ (the programs are described in \cite{FS19add}), and disproved it in $\F_2^d$ for $d\ge10$ by dephasing Tao's Hadamard matrix of order $12$ to a log-Hadamard matrix of binary rank ten. Their status summary leaves the binary dimensions $d=7,8,9$ unresolved \cite[Section~1]{FS20}, and the same frontier is recorded in the later account of Malikiosis \cite{Mal24}. In the other direction, Fan and Kadir \cite{FK26} have recently constructed a tile of $\F_2^{13}$ without a spectrum, so that the implication from tiling to spectrality fails in $\F_2^d$ for $d\ge13$; see also the periodic tiling approach of Fan and Zhang \cite{FZ24}. Our main result settles the case $d=7$.

\begin{theorem}\label{thm:main}
Let $E$ be a nonempty subset of $\F_2^7$. Then $E$ is spectral if and only if there is $T\subseteq\F_2^7$ with $E\oplus T=\F_2^7$.
\end{theorem}

\emph{Tiles are spectral.} Tiles of $\F_2^7$ have power-of-two size. Sets of size at most eight are treated by elementary projections and, in affine dimension at most six, by the theorem of Ferguson and Sothanaphan. For a tile $E$ of size at least $16$ the tiling complement $T$ has at most eight elements, and it suffices to find a common spectrum for all tiling complements of $T$. The following theorem provides one, via the sufficient condition of Lagarias and Szab\'o \cite{LS01} in the form of \cite[Proposition~4]{FMM06}.

\begin{theorem}\label{thm:universal-intro}
Let $d\ge3$ and let $T\subseteq\F_2^d$ have eight elements. There is a set $B\subseteq\F_2^d$ of size $2^{d-3}$ such that $\sum_{t\in T}(-1)^{(b+b')\cdot t}\ne0$ for all distinct $b,b'\in B$.
\end{theorem}

The construction starts from a simplex in $\F_2^7$ together with its antipodal translate and then passes to a coset of a dual subspace (Theorem~\ref{thm:universal}). Together with elementary analogues for sets of one, two and four elements, it shows that every tile of $\F_2^d$ with at least $2^{d-3}$ elements is spectral (Corollary~\ref{cor:large-tiles}).

\emph{Spectral sets tile.} A spectral pair of size $m$ in $\F_2^d$ yields an $m\times m$ minor of the Walsh matrix that is a Hadamard matrix. Comparing the $2$-adic valuation of its determinant with the determinantal divisors of the Walsh matrix, we exclude every size that is not a power of two in $\F_2^7$ (and also in $\F_2^8$; see Remark~\ref{rem:smith-other}). The sizes $64$ and $128$ are handled by an elementary argument \cite{AAB17}. For the sizes $16$ and $32$ we attach to the dephased log-Hadamard matrix of a spectral pair a \emph{canonical column set} in $\F_2^r$, where $r\le7$ is the binary rank, and show that the original spectral set projects bijectively onto it (Lemma~\ref{lem:canonical}). A tiling of the canonical column set then lifts to a tiling of the original set. Canonical column sets of equivalent Hadamard matrices are affinely equivalent (Lemma~\ref{lem:equivalence}), so it suffices to treat one representative of each equivalence class, in both orientations when transposes are identified. For order $16$ we use Hall's classification \cite{Hal61}; for order $32$ we use the classification of Kharaghani and Tayfeh-Rezaie \cite{KT13} in the form distributed by McKay \cite{McK}. Scanning all $6{,}857{,}010$ extended classes of order $32$, we find exactly sixteen classes of dephased rank at most seven, and all thirty-two resulting canonical column sets tile. As a by-product, every sixteen-point spectral set in any binary vector space tiles (Corollary~\ref{cor:sixteen}).

\emph{Dependencies.} Besides standard facts from \cite{AAB17}, the proof uses three results that we do not reprove: the theorem of Ferguson and Sothanaphan in dimension at most six, applied only to eight-point sets; Hall's classification of Hadamard matrices of order $16$; and the classification of Hadamard matrices of order $32$ by Kharaghani and Tayfeh-Rezaie, together with the correctness of the data file that represents it. The classification of order $32$ enters only through Proposition~\ref{prop:smallrank}, and Theorem~\ref{thm:universal} and the cardinality bound of Section~\ref{sec:cardinality} use no classification. To our knowledge, Theorem~\ref{thm:main} has not appeared in the literature.

\emph{Organization.} Section~\ref{sec:prelim} collects standard facts and the projection lemma. Section~\ref{sec:tiles} proves Theorem~\ref{thm:universal-intro} and the tile-to-spectral implication. Section~\ref{sec:cardinality} treats cardinalities. Section~\ref{sec:hadamard} introduces canonical column sets, and Section~\ref{sec:computation} describes the finite computations. Section~\ref{sec:proof} completes the proof and discusses the next dimension. The appendices list the data needed to check the computations.

\section{Preliminaries}\label{sec:prelim}

Throughout, $G=\F_2^d$, and $u\cdot x$ is the standard bilinear form. We identify $G$ with its character group via $\chi_u(x)=(-1)^{u\cdot x}$. For $A\subseteq G$ and $u\in G$ put
\[
W_A(u)=\sum_{a\in A}(-1)^{u\cdot a}.
\]
This is the unnormalized Fourier transform of the indicator function $\one_A$; it is integer valued. Translating $A$ by $c$ multiplies $W_A(u)$ by $(-1)^{u\cdot c}$, so the zero set of $W_A$ is translation invariant. Since $G$ has exponent two, $A-A=A+A$.

A set $B\subseteq G$ is a \emph{spectrum} of a nonempty set $A$ if $|B|=|A|$ and the characters $\chi_b$, $b\in B$, are pairwise orthogonal on $A$, that is,
\begin{equation}\label{eq:spectral}
W_A(b+b')=0\qquad\text{for all distinct } b,b'\in B .
\end{equation}
Then $(A,B)$ is a \emph{spectral pair}. The matrix $H=(\chi_b(a))_{b\in B,\,a\in A}$ has entries $\pm1$ and satisfies $HH^{\T}=|A|I$, hence also $H^{\T}H=|A|I$; thus $(B,A)$ is a spectral pair as well, and $H$ is a real Hadamard matrix \cite[Theorem~4.1]{AAB17}. Translating $A$ and $B$ independently preserves spectral pairs.

For a tiling pair $A\oplus T=G$ the convolution identity $\one_A*\one_T=\one_G$ gives $|A|\,|T|=|G|$ and
\begin{equation}\label{eq:tiling}
W_A(u)\,W_T(u)=0\qquad(u\ne0).
\end{equation}
Conversely, $|A|\,|T|=|G|$ together with \eqref{eq:tiling}, or together with $(A+A)\cap(T+T)=\{0\}$, implies $A\oplus T=G$ \cite[Theorem~3.1]{AAB17}.

\begin{lemma}\label{lem:affine}
Translations and invertible linear maps of $G$ preserve tiling and spectrality. If $A$ is contained in a linear subspace $V\subseteq G$, then $A$ tiles $G$ if and only if $A$ tiles $V$, and $A$ is spectral in $G$ if and only if $A$ is spectral in $V$.
\end{lemma}

\begin{proof}
The first statement is \cite[Corollaries~3.2 and 4.3]{AAB17}; an invertible linear map $M$ acts on frequencies by $(M^{\T})^{-1}$. Let $A\subseteq V$. If $A\oplus T=G$, then every $v\in V$ is $a+t$ with $t=v+a\in V$, so $A\oplus(T\cap V)=V$. Conversely, if $A\oplus T_V=V$ and $K$ is a linear complement of $V$, then $A\oplus(T_V+K)=G$. Every character of $V$ extends to $G$, so a spectrum in $V$ yields one in $G$. If $B$ is a spectrum of $A$ in $G$, two distinct elements of $B$ cannot induce the same character of $V$, since their characters would then agree on $A$ and could not be orthogonal there. Hence restriction to $V$ produces $|A|$ distinct, pairwise orthogonal characters of $V$.
\end{proof}

The next lemma is due to Szegedy \cite{Sze98}; see \cite[Lemma~6]{FMM06}. Its second part is the statement that graph sets are spectral \cite[Proposition~6.2]{AAB17}.

\begin{lemma}[Projection and lifting]\label{lem:lift}
Let $q\colon G\to V$ be a surjective linear map onto a binary vector space $V$, and let $E\subseteq G$ be a set on which $q$ is injective. If $q(E)\oplus T=V$, then $E\oplus q^{-1}(T)=G$. If $q(E)=V$, then $E$ is also spectral.
\end{lemma}

\begin{proof}
Let $g\in G$. There is exactly one pair $(c,t)\in q(E)\times T$ with $q(g)=c+t$, and exactly one $e\in E$ with $q(e)=c$. Then $g+e\in q^{-1}(T)$, and the decomposition $g=e+(g+e)$ is the only one with first summand in $E$ and second in $q^{-1}(T)$. If $q(E)=V$, the characters $\psi\circ q$, $\psi$ a character of $V$, are $|V|=|E|$ distinct characters of $G$; for $\psi\ne\psi'$ we have
\[
\sum_{e\in E}(\psi\psi')(q(e))=\sum_{v\in V}(\psi\psi')(v)=0.
\]
\end{proof}

\begin{lemma}\label{lem:small}
Every subset of $G$ with one, two or four elements tiles $G$ and is spectral.
\end{lemma}

\begin{proof}
After a translation the set contains $0$. A set with one or two elements is a subspace. A four-point set either is an affine plane, hence a coset of a subspace, or is affinely equivalent to $\{0,e_1,e_2,e_3\}$. In the latter case a linear map $q\colon G\to\F_2^2$ with $q(e_1)=(1,0)$, $q(e_2)=(0,1)$ and $q(e_3)=(1,1)$ maps the set bijectively onto $\F_2^2$, and Lemma~\ref{lem:lift} applies (cf.\ \cite[Proposition~5.1]{FS20}).
\end{proof}

\begin{lemma}\label{lem:eight}
Let $d\ge3$ and let $E\subseteq\F_2^d$ have eight elements. Then $E$ is spectral if and only if $E$ tiles $\F_2^d$.
\end{lemma}

\begin{proof}
Translate $E$ to contain $0$ and let $V$ be its linear span, of dimension $r\le7$. If $r\le6$, Lemma~\ref{lem:affine} reduces the statement to $V\cong\F_2^r$, where it is part of \cite[Theorem~1.2]{FS20}. If $r=7$, then $E=\{0,v_1,\dots,v_7\}$ with $v_1,\dots,v_7$ independent. A linear map $q\colon\F_2^d\to\F_2^3$ sending $v_1,\dots,v_7$ to the seven nonzero vectors of $\F_2^3$ is surjective and bijective on $E$, so Lemma~\ref{lem:lift} shows that $E$ tiles and is spectral.
\end{proof}

\begin{remark}\label{rem:clvz}
Cohen, Litsyn, Vardy and Z\'emor showed that every tile of $\F_2^d$ with at most eight elements is a linear tile, that is, it has a tiling complement that is a linear subspace \cite[end of Section~4; Corollary~7.3 and Proposition~6.1]{CLVZ96}. If $A$ is such a complement of a tile $E$, the quotient map $\F_2^d\to\F_2^d/A$ maps $E$ bijectively onto $\F_2^d/A$, so $E$ is spectral by Lemma~\ref{lem:lift}.
\end{remark}

\section{Universal spectra for small tiling complements}\label{sec:tiles}

Following Lagarias and Wang \cite{LW97}, a \emph{universal spectrum} of a tile $T$ is a set that is a spectrum of every tiling complement of $T$. If $|B|\,|T|=|G|$ and $W_T(b+b')\ne0$ for all distinct $b,b'\in B$, then $B$ is a universal spectrum of $T$; this sufficient condition is due to Lagarias and Szab\'o \cite{LS01}, see \cite[Proposition~4]{FMM06}. Indeed, if $E\oplus T=G$, then \eqref{eq:tiling} forces $W_E(b+b')=0$ for distinct $b,b'\in B$, and $|B|=|E|$. The condition does not require $T$ to tile.

\begin{theorem}\label{thm:universal}
Let $d\ge3$ and let $T\subseteq\F_2^d$ have eight elements. There is a set $B\subseteq\F_2^d$ with $|B|=2^{d-3}$ such that
\begin{equation}\label{eq:nonzero}
W_T(b+b')\ne0\qquad\text{for all distinct } b,b'\in B .
\end{equation}
Consequently, if $E\oplus T=\F_2^d$, then $B$ is a spectrum of $E$.
\end{theorem}

\begin{proof}
Since \eqref{eq:nonzero} is translation invariant in $T$, we may assume that $0\in T$ and write $T=\{0,t_1,\dots,t_7\}$. Let $V$ be the span of $T$ and $r=\dim V$; as $|T|=8$, we have $3\le r\le 7$. Write $V^*$ for the dual space of $V$, and for $\alpha\in V^*$ put
\[
\widetilde W_T(\alpha)=\sum_{t\in T}(-1)^{\alpha(t)}.
\]

Define $\phi^*\colon V^*\to\F_2^7$ by $\phi^*(\alpha)=(\alpha(t_1),\dots,\alpha(t_7))$. It is linear and injective, because $t_1,\dots,t_7$ span $V$; let $U$ be its image, an $r$-dimensional subspace of $\F_2^7$. With $S=\{0,e_1,\dots,e_7\}\subseteq\F_2^7$ we have
\begin{equation}\label{eq:pullback}
\widetilde W_T(\alpha)=W_S(\phi^*\alpha)\qquad(\alpha\in V^*),
\end{equation}
and $W_S(x)=1+\sum_{i=1}^7(-1)^{x_i}=8-2\wt(x)$, where $\wt$ is the Hamming weight. Thus $W_S(x)=0$ exactly when $\wt(x)=4$.

Let $j=(1,\dots,1)\in\F_2^7$ and $B_0=S\cup(j+S)$. The two parts are disjoint, so $|B_0|=16$. The sum of two distinct points of $S$, or of $j+S$, has weight $1$ or $2$; the sum of a point of $S$ and a point of $j+S$ has weight $5$, $6$ or $7$. Hence $W_S(x+y)\ne0$ for all distinct $x,y\in B_0$.

The $2^{7-r}$ cosets of $U$ partition $\F_2^7$, so some coset $c+U$ contains at least $16/2^{7-r}=2^{r-3}$ points of $B_0$. Choose $k=2^{r-3}$ of them, say $x_1,\dots,x_k$, and put $B_V=(\phi^*)^{-1}(\{x_1+x_i:1\le i\le k\})$. Then $|B_V|=2^{r-3}$, and for distinct $\alpha,\alpha'\in B_V$ the vector $\phi^*(\alpha+\alpha')$ is a sum of two distinct points of $B_0$; by \eqref{eq:pullback}, $\widetilde W_T(\alpha+\alpha')\ne0$.

Finally let $\rho\colon\F_2^d\to V^*$ be the surjection $\rho(u)(v)=u\cdot v$, and put $B=\rho^{-1}(B_V)$. Then $|B|=2^{r-3}\cdot2^{d-r}=2^{d-3}$, and $W_T(b+b')=\widetilde W_T(\rho(b)+\rho(b'))$ for $b,b'\in\F_2^d$. If $\rho(b)\ne\rho(b')$ this is nonzero by the choice of $B_V$; if $\rho(b)=\rho(b')$ it equals $|T|=8$. This proves \eqref{eq:nonzero}. The last assertion follows from the discussion preceding the theorem.
\end{proof}

\begin{remark}\label{rem:universal}
The set $B$ is a spectrum for every tiling complement of $T$, not necessarily for $T$ itself. By \cite[Proposition~4]{FMM06}, $B$ is also a common tiling complement of all spectra of $T$: if $L$ is a spectrum of $T$, then $(L+L)\setminus\{0\}$ lies in the zero set of $W_T$, which is disjoint from $(B+B)\setminus\{0\}$, and $|L|\,|B|=2^d$. Thus every spectrum of an eight-point spectral subset of $\F_2^d$ tiles $\F_2^d$.
\end{remark}

\begin{lemma}\label{lem:universal-small}
Let $T\subseteq\F_2^d$ have one, two or four elements. There is a set $B\subseteq\F_2^d$ with $|B|=2^d/|T|$ satisfying \eqref{eq:nonzero}.
\end{lemma}

\begin{proof}
As in the last step of the proof of Theorem~\ref{thm:universal}, it suffices to find $B_V\subseteq V^*$ with $|B_V|=2^r/|T|$ and $\widetilde W_T(\alpha+\alpha')\ne0$ for distinct $\alpha,\alpha'\in B_V$, where $0\in T$, $V$ is the span of $T$ and $r=\dim V$. If $T=V$, take $B_V=\{0\}$. Otherwise $|T|=4$ and $T=\{0,t_1,t_2,t_3\}$ with $t_1,t_2,t_3$ a basis of $V$; take $B_V=\{0,\alpha\}$ with $\alpha(t_i)=1$ for all $i$, so that $\widetilde W_T(\alpha)=1-3=-2$.
\end{proof}

\begin{corollary}\label{cor:large-tiles}
Let $d\ge3$. Every tile $E$ of $\F_2^d$ with $|E|\ge2^{d-3}$ is spectral.
\end{corollary}

\begin{proof}
A tiling complement $T$ of $E$ has $|T|=2^d/|E|\in\{1,2,4,8\}$. By Theorem~\ref{thm:universal} or Lemma~\ref{lem:universal-small} there is $B$ with $|B|=|E|$ satisfying \eqref{eq:nonzero}, and \eqref{eq:tiling} shows that $B$ is a spectrum of $E$.
\end{proof}

\begin{proposition}\label{prop:forward}
Every tile of $\F_2^7$ is spectral.
\end{proposition}

\begin{proof}
The size of a tile divides $128$. Sizes $1$, $2$ and $4$ follow from Lemma~\ref{lem:small}, size $8$ from Lemma~\ref{lem:eight}, and sizes $16$ to $128$ from Corollary~\ref{cor:large-tiles}.
\end{proof}

\section{Cardinalities of spectral sets}\label{sec:cardinality}

Let $S_d$ be the $2^d\times2^d$ Walsh matrix with entries $S_d(x,y)=(-1)^{x\cdot y}$, $x,y\in\F_2^d$. Let $\mathcal{E}_d$ be the multiset in which each $k\in\{0,\dots,d\}$ occurs $\binom dk$ times, and let $s_d(m)$ be the sum of the $m$ smallest elements of $\mathcal{E}_d$. For an integer $n\ne0$, $\nu_2(n)$ denotes the $2$-adic valuation of $n$.

\begin{lemma}\label{lem:smith}
The Smith normal form of $S_d$ over $\Z$ has diagonal entries $2^k$ with multiplicity $\binom dk$. Every $m\times m$ minor of $S_d$ is divisible by $2^{s_d(m)}$.
\end{lemma}

\begin{proof}
With
\[
U_1=\begin{pmatrix}1&0\\1&-1\end{pmatrix},\qquad
D_1=\begin{pmatrix}1&0\\0&2\end{pmatrix},\qquad
V_1=\begin{pmatrix}1&1\\0&1\end{pmatrix}
\]
we have $S_1=U_1D_1V_1$, and $\det U_1=-1$, $\det V_1=1$. Taking $d$-fold Kronecker powers, $S_d=U_dD_dV_d$ with $U_d,V_d$ invertible over $\Z$ and $D_d=\diag(2^{\wt(x)})_{x\in\F_2^d}$. Sorting the diagonal entries gives a divisibility chain, which proves the first statement \cite{New72}. By the Cauchy--Binet formula, every $m\times m$ minor of $S_d$ is an integral linear combination of $m\times m$ minors of $D_d$. Each such minor is $0$ or a product of $m$ distinct diagonal entries, and is therefore divisible by $2^{s_d(m)}$.
\end{proof}

\begin{proposition}\label{prop:detbound}
If $E\subseteq\F_2^d$ is spectral and $m=|E|\ge2$, then
\begin{equation}\label{eq:detbound}
\frac m2\,\nu_2(m)\ \ge\ s_d(m).
\end{equation}
\end{proposition}

\begin{proof}
If $B$ is a spectrum of $E$, the matrix $H=(\chi_b(e))_{b\in B,\,e\in E}$ is an $m\times m$ submatrix of $S_d$ with $HH^{\T}=mI$. Hence $(\det H)^2=m^m$, so $\nu_2(\det H)=\frac m2\nu_2(m)$, and Lemma~\ref{lem:smith} gives \eqref{eq:detbound}.
\end{proof}

For $d=7$ the multiplicities in $\mathcal{E}_7$ are $1,7,21,35,35,21,7,1$. Table~\ref{tab:smith} lists the two sides of \eqref{eq:detbound} for every multiple of four below $64$ that is not a power of two.

\begin{table}[ht]
\centering
\caption{The bound \eqref{eq:detbound} for $d=7$.}\label{tab:smith}
\begin{tabular}{lrrrrrrrrrrr}
\toprule
$m$ & $12$ & $20$ & $24$ & $28$ & $36$ & $40$ & $44$ & $48$ & $52$ & $56$ & $60$\\
\midrule
$\frac m2\nu_2(m)$ & $12$ & $20$ & $36$ & $28$ & $36$ & $60$ & $44$ & $96$ & $52$ & $84$ & $60$\\
$s_7(m)$ & $15$ & $31$ & $39$ & $47$ & $70$ & $82$ & $94$ & $106$ & $118$ & $130$ & $142$\\
\bottomrule
\end{tabular}
\end{table}

The following facts are \cite[Corollary~4.5 and Proposition~6.4]{AAB17}; we include the short argument.

\begin{lemma}\label{lem:half}
Let $E\subseteq\F_2^d$ be spectral. If $|E|>2^{d-1}$, then $E=\F_2^d$. If $|E|=2^{d-1}$, then $E\oplus\{0,t\}=\F_2^d$ for some $t$.
\end{lemma}

\begin{proof}
Let $B$ be a spectrum of $E$. If $|B|>2^{d-1}$, then $B$ and $B+u$ intersect for every $u$, so every $u\ne0$ is a sum of two distinct elements of $B$, and \eqref{eq:spectral} gives $W_E(u)=0$ for all $u\ne0$. By Fourier inversion $\one_E$ is constant, so $E=\F_2^d$. If $|E|=2^{d-1}$ and every $u\ne0$ were a sum of two distinct elements of $E$, the same argument applied to the spectral pair $(B,E)$ would give $B=\F_2^d$, which is impossible. Hence there is $t\ne0$ outside $E+E$; then $E$ and $E+t$ are disjoint and together have $2^d$ elements.
\end{proof}

\begin{corollary}\label{cor:cardinality}
A nonempty spectral subset of $\F_2^7$ has $1$, $2$, $4$, $8$, $16$, $32$, $64$ or $128$ elements. Every spectral subset of $\F_2^7$ with $64$ elements tiles.
\end{corollary}

\begin{proof}
As is well known, a real Hadamard matrix has order $1$, $2$ or a multiple of four: if the first row is normalized to $(1,\dots,1)$, the columns split into two halves according to the signs of the second row, and orthogonality of a third row to the first two forces it to be balanced on each half, so that each half has even size. By Proposition~\ref{prop:detbound} and Table~\ref{tab:smith}, no spectral set has size $m\le64$ that is a multiple of four but not a power of two. For sizes at least $64$ we apply Lemma~\ref{lem:half}.
\end{proof}

\begin{remark}\label{rem:smith-other}
The same computation applies in other dimensions. For $d=8$, inequality \eqref{eq:detbound} fails for every multiple $m\le128$ of four that is not a power of two; together with Lemma~\ref{lem:half}, every spectral subset of $\F_2^8$ has power-of-two size. For $d=9$ it fails for all such $m\le256$ except $m=48$ and $m=96$. For $m=12$ the inequality holds exactly when $d\ge10$; by \cite[Theorem~4.7]{AAB17} this means that every binary log-Hadamard matrix of order $12$ has rank at least ten, and the rank-ten matrix of \cite[Section~5]{FS20} shows that the bound is attained. In dimensions five and six, the sizes $20$, $24$ and $28$ were excluded in \cite{FS19add} using Sloane's library of Hadamard matrices \cite{Slo}.
\end{remark}

\section{Canonical column sets of log-Hadamard matrices}\label{sec:hadamard}

All ranks in this section are over $\F_2$. A binary $m\times m$ matrix $L$ is \emph{log-Hadamard} if $\bigl((-1)^{L_{ij}}\bigr)$ is a real Hadamard matrix, and \emph{dephased} if its first row and first column vanish. For a real Hadamard matrix $H$ of order $m$ and indices $i_0,j_0$, the \emph{dephasing of $H$ at $(i_0,j_0)$} is the binary matrix $L$ with
\[
(-1)^{L_{ij}}=H_{ij}H_{ij_0}H_{i_0j}H_{i_0j_0}.
\]
It is log-Hadamard, being obtained from $H$ by multiplying rows and columns by signs, and its row $i_0$ and column $j_0$ vanish. The dephasing of $H$ at the first row and column is simply called the dephasing of $H$, and its rank the \emph{dephased rank} of $H$. Rows and columns are indexed from $0$.

\begin{definition}\label{def:canonical}
Let $L$ be a dephased log-Hadamard matrix of order $m\ge2$ and rank $r$. Let $C$ be an $r\times m$ matrix whose rows form a basis of the row space of $L$. The set of columns of $C$, viewed as points of $\F_2^r$, is called a \emph{canonical column set} of $L$.
\end{definition}

The $m$ columns of $C$ are distinct: two equal columns of $C$ would give two equal columns of $L$, contradicting the orthogonality of the corresponding columns of the Hadamard matrix. A different choice of basis changes a canonical column set by an element of $\GL(r,\F_2)$. Canonical column sets make the construction in \cite[Theorem~4.7]{AAB17} explicit. The next lemma shows that the original spectral set projects bijectively onto them.

\begin{lemma}\label{lem:canonical}
Let $(E,B)$ be a spectral pair in $\F_2^d$ with $0\in E$ and $0\in B$, and let $L=(b\cdot e)_{b\in B,\,e\in E}$, with rows and columns ordered so that $0$ comes first. Then $L$ is a dephased log-Hadamard matrix. If $r=\rank L$, there is a surjective linear map $q\colon\F_2^d\to\F_2^r$ which maps $E$ bijectively onto a canonical column set of $L$. In particular, if some canonical column set of $L$ tiles $\F_2^r$, then $E$ tiles $\F_2^d$.
\end{lemma}

\begin{proof}
The matrix $((-1)^{b\cdot e})$ is the Hadamard matrix of the spectral pair, and the row and column of $0$ vanish. Choose $b_1,\dots,b_r\in B$ whose rows in $L$ form a basis of the row space, and define $q(x)=(b_1\cdot x,\dots,b_r\cdot x)$. The vectors $b_1,\dots,b_r$ are linearly independent, since a linear relation among them would give the same relation among the corresponding rows of $L$; hence $q$ is surjective. For $e\in E$, $q(e)$ is the column of $C$ indexed by $e$, where $C$ consists of the chosen rows. Thus $q(E)$ is a canonical column set, and $q$ is injective on $E$ because the columns of $C$ are distinct. If some canonical column set tiles, so does $q(E)$, which differs from it by an invertible linear map; Lemma~\ref{lem:lift} completes the proof.
\end{proof}

Two real Hadamard matrices are \emph{equivalent} if one is obtained from the other by permuting rows, permuting columns, and multiplying rows and columns by $-1$. Transposition is not part of this equivalence. The rank statement in the next lemma is \cite[Proposition~4.3]{FS20}; we need the finer statement about the point sets.

\begin{lemma}\label{lem:equivalence}
Let $H$ and $H'$ be equivalent real Hadamard matrices of order $m\ge2$. Then $H$ and $H'$ have the same dephased rank $r$, and every canonical column set of the dephasing of $H$ is affinely equivalent, inside $\F_2^r$, to every canonical column set of the dephasing of $H'$. In particular, one of them tiles $\F_2^r$ if and only if the other does.
\end{lemma}

\begin{proof}
Let $L$ be the dephasing of $H$. The sign changes cancel in the defining product, so the dephasing of $H'$ coincides, up to permutations of rows and columns, with the dephasing $L'$ of $H$ at some $(i_0,j_0)$. A direct check of signs gives
\begin{equation}\label{eq:redephase}
L'_{ij}=L_{ij}+L_{ij_0}+L_{i_0j}+L_{i_0j_0}.
\end{equation}
Let $r=\rank L$ and write $L=PC$ with $P$ of size $m\times r$ and $C$ of size $r\times m$, both of rank $r$; denote the rows of $P$ by $p_i$ and the columns of $C$ by $c_j$, so $L_{ij}=p_i\cdot c_j$. Since row $0$ of $L$ vanishes and the $c_j$ span $\F_2^r$, we have $p_0=0$; similarly $c_0=0$. Then \eqref{eq:redephase} becomes
\[
L'_{ij}=(p_i+p_{i_0})\cdot(c_j+c_{j_0}).
\]
The points $c_j+c_{j_0}$ span $\F_2^r$: they include $c_0+c_{j_0}=c_{j_0}$ and hence every $c_j$. The same holds for the points $p_i+p_{i_0}$. Therefore $\rank L'=r$, and the rows of $L'$ are the vectors $(p_i+p_{i_0})^{\T}C'$, where $C'$ has columns $c_j+c_{j_0}$. A basis of the row space of $L'$ is $MC'$ for an invertible $r\times r$ matrix $M$, so the corresponding canonical column set is the image of $\{c_j+c_{j_0}\}$ under $M$. This set is affinely equivalent to the canonical column set $\{c_j\}$ of $L$, and the lemma follows since permutations of rows and columns do not change the set of columns.
\end{proof}

\begin{remark}\label{rem:transpose}
The dephasing of $H^{\T}$ is the transpose of the dephasing of $H$, so $H$ and $H^{\T}$ have the same dephased rank. In the notation of the proof above, a canonical column set of the dephasing of $H^{\T}$ is affinely equivalent to $\{p_i\}$, which in general is not affinely equivalent to $\{c_j\}$. Consequently, a classification modulo equivalence and transposition requires both orientations of each representative to be checked.
\end{remark}

\section{The finite computations}\label{sec:computation}

\subsection{Sources}
Hall \cite{Hal61} showed that there are exactly five equivalence classes of Hadamard matrices of order $16$; we use the representatives \texttt{had.16.0}, \dots, \texttt{had.16.4} of Sloane's library \cite{Slo}. Kharaghani and Tayfeh-Rezaie \cite{KT13} completed the classification of Hadamard matrices of order $32$, which has $13{,}710{,}027$ equivalence classes. McKay's database \cite{McK} contains representatives of all these classes. It also contains representatives of the $6{,}857{,}010$ \emph{extended} classes, that is, equivalence classes up to transposition, and of the $3{,}993$ classes that are invariant under transposition. We process the file of extended classes (see Appendix~\ref{app:data}).

Each record of the file consists of $32$ hexadecimal words, one for each row. The $(i,j)$ entry is $+1$ if bit $j$ of word $i$ is $1$ and $-1$ otherwise, where bit $0$ is the least significant bit. Records are numbered from $1$ in the order of the file.

\subsection{Order 16}
The dephased ranks of the five representatives of order $16$ are $4,5,6,7,7$; the class of rank four is the Sylvester class, whose canonical column set is all of $\F_2^4$. Table~\ref{tab:sixteen} lists, for each representative $H$, a canonical column set $E_C$ of the dephasing of $H$, in the order of the columns of $H$, and a tiling complement $T$. Here and below an integer $x$ denotes the vector of $\F_2^r$ whose $i$-th coordinate is bit $i$ of $x$.

\begin{table}[ht]
\centering
\caption{Canonical column sets and tiling complements for the five classes of order $16$.}\label{tab:sixteen}
\small
\begin{tabular}{@{}l c >{\raggedright\arraybackslash}p{0.52\textwidth} >{\raggedright\arraybackslash}p{0.18\textwidth}@{}}
\toprule
$H$ & $r$ & $E_C$ & $T$\\
\midrule
\texttt{had.16.0} & 4 & $(0,\allowbreak 1,\allowbreak 2,\allowbreak 3,\allowbreak 4,\allowbreak 5,\allowbreak 6,\allowbreak 7,\allowbreak 8,\allowbreak 9,\allowbreak 10,\allowbreak 11,\allowbreak 12,\allowbreak 13,\allowbreak 14,\allowbreak 15)$ & $\{0\}$ \\
\texttt{had.16.1} & 5 & $(0,\allowbreak 17,\allowbreak 2,\allowbreak 19,\allowbreak 4,\allowbreak 21,\allowbreak 22,\allowbreak 7,\allowbreak 24,\allowbreak 9,\allowbreak 26,\allowbreak 11,\allowbreak 28,\allowbreak 13,\allowbreak 14,\allowbreak 31)$ & $\{0,\allowbreak 25\}$ \\
\texttt{had.16.2} & 6 & $(0,\allowbreak 33,\allowbreak 18,\allowbreak 51,\allowbreak 4,\allowbreak 53,\allowbreak 38,\allowbreak 23,\allowbreak 56,\allowbreak 25,\allowbreak 42,\allowbreak 11,\allowbreak 60,\allowbreak 13,\allowbreak 30,\allowbreak 47)$ & $\{0,\allowbreak 41,\allowbreak 40,\allowbreak 57\}$ \\
\texttt{had.16.3} & 7 & $(0,\allowbreak 65,\allowbreak 34,\allowbreak 19,\allowbreak 100,\allowbreak 53,\allowbreak 118,\allowbreak 87,\allowbreak 120,\allowbreak 57,\allowbreak 90,\allowbreak 107,\allowbreak 28,\allowbreak 77,\allowbreak 14,\allowbreak 47)$ & $\{0,\allowbreak 76,\allowbreak 74,\allowbreak 85,\allowbreak 83,\allowbreak 103,\allowbreak 97,\allowbreak 126\}$ \\
\texttt{had.16.4} & 7 & $(0,\allowbreak 21,\allowbreak 38,\allowbreak 51,\allowbreak 56,\allowbreak 45,\allowbreak 30,\allowbreak 11,\allowbreak 64,\allowbreak 116,\allowbreak 122,\allowbreak 78,\allowbreak 105,\allowbreak 93,\allowbreak 83,\allowbreak 103)$ & $\{0,\allowbreak 63,\allowbreak 61,\allowbreak 47,\allowbreak 62,\allowbreak 57,\allowbreak 59,\allowbreak 17\}$ \\
\bottomrule
\end{tabular}
\end{table}

\subsection{Order 32: the rank scan}
For every record of the extended file we computed the dephased rank by Gaussian elimination over $\F_2$. Three independently written programs processed all $6{,}857{,}010$ records and selected the same sixteen records of rank at most seven; Table~\ref{tab:ranks} shows the full distribution.

\begin{table}[ht]
\centering
\caption{Dephased binary ranks of the $6{,}857{,}010$ extended classes of order $32$.}\label{tab:ranks}
\small
\begin{tabular}{lrrrrrr}
\toprule
rank & $5$ & $6$ & $7$ & $8$ & $9$ & $10$\\
classes & $1$ & $2$ & $13$ & $150$ & $1{,}874$ & $20{,}202$\\
\midrule
rank & $11$ & $12$ & $13$ & $14$ & $15$ & total\\
classes & $158{,}127$ & $786{,}379$ & $2{,}250{,}131$ & $2{,}753{,}834$ & $886{,}297$ & $6{,}857{,}010$\\
\bottomrule
\end{tabular}
\end{table}

\begin{table}[ht]
\centering
\caption{The sixteen extended classes of order $32$ with dephased rank at most seven, and tiling complements for both orientations. The index is the record number in the extended file.}\label{tab:witnesses}
\small
\begin{tabular}{rclll}
\toprule
index & rank & complement for $H$ & complement for $H^{\T}$\\
\midrule
25418 & 5 & $\{0\}$ & $\{0\}$ \\
25419 & 6 & $\{0,\allowbreak 56\}$ & $\{0,\allowbreak 60\}$ \\
25420 & 7 & $\{0,\allowbreak 120,\allowbreak 92,\allowbreak 60\}$ & $\{0,\allowbreak 124,\allowbreak 104,\allowbreak 60\}$ \\
25421 & 6 & $\{0,\allowbreak 60\}$ & $\{0,\allowbreak 60\}$ \\
25422 & 7 & $\{0,\allowbreak 124,\allowbreak 88,\allowbreak 52\}$ & $\{0,\allowbreak 116,\allowbreak 84,\allowbreak 48\}$ \\
25423 & 7 & $\{0,\allowbreak 124,\allowbreak 88,\allowbreak 52\}$ & $\{0,\allowbreak 124,\allowbreak 96,\allowbreak 60\}$ \\
25428 & 7 & $\{0,\allowbreak 124,\allowbreak 88,\allowbreak 52\}$ & $\{0,\allowbreak 126,\allowbreak 108,\allowbreak 90\}$ \\
25432 & 7 & $\{0,\allowbreak 124,\allowbreak 92,\allowbreak 60\}$ & $\{0,\allowbreak 126,\allowbreak 106,\allowbreak 60\}$ \\
25434 & 7 & $\{0,\allowbreak 124,\allowbreak 88,\allowbreak 52\}$ & $\{0,\allowbreak 126,\allowbreak 108,\allowbreak 90\}$ \\
25438 & 7 & $\{0,\allowbreak 124,\allowbreak 90,\allowbreak 54\}$ & $\{0,\allowbreak 126,\allowbreak 108,\allowbreak 58\}$ \\
25467 & 7 & $\{0,\allowbreak 126,\allowbreak 94,\allowbreak 56\}$ & $\{0,\allowbreak 126,\allowbreak 94,\allowbreak 56\}$ \\
25472 & 7 & $\{0,\allowbreak 126,\allowbreak 90,\allowbreak 54\}$ & $\{0,\allowbreak 126,\allowbreak 92,\allowbreak 58\}$ \\
25475 & 7 & $\{0,\allowbreak 126,\allowbreak 80,\allowbreak 62\}$ & $\{0,\allowbreak 126,\allowbreak 110,\allowbreak 56\}$ \\
25485 & 7 & $\{0,\allowbreak 124,\allowbreak 88,\allowbreak 52\}$ & $\{0,\allowbreak 124,\allowbreak 110,\allowbreak 90\}$ \\
25487 & 7 & $\{0,\allowbreak 124,\allowbreak 88,\allowbreak 52\}$ & $\{0,\allowbreak 124,\allowbreak 110,\allowbreak 90\}$ \\
44186 & 7 & $\{0,\allowbreak 96,\allowbreak 120,\allowbreak 24\}$ & $\{0,\allowbreak 125,\allowbreak 108,\allowbreak 49\}$ \\
\bottomrule
\end{tabular}
\end{table}

\subsection{Order 32: tiling certificates}
For each of the sixteen retained records and each orientation $H$ and $H^{\T}$ we computed a canonical column set and a tiling complement. A certificate consists of the rank $r$, a list $E_C$ of $32$ points of $\F_2^r$ in column order, and a set $T$. Two independently written checkers verify, using only integer arithmetic, that
\begin{enumerate}
\item the source matrix has entries $\pm1$ and satisfies $HH^{\T}=32I$;
\item its dephasing has rank $r$;
\item the $r\times32$ matrix whose columns are the points of $E_C$ has rank $r$ and the same row space as the dephasing, so that $E_C$ is a canonical column set of it;
\item $E_C$ has $32$ distinct points, $T$ has $2^r/32$ distinct points, and the $2^r$ sums $e+t$ are distinct.
\end{enumerate}
The same checks are applied to the five certificates of order $16$. All thirty-two certificates of order $32$ satisfy these conditions; Table~\ref{tab:witnesses} lists the complements, and Appendix~\ref{app:certificates} lists the matrices and the point sets.

\begin{proposition}\label{prop:smallrank}
Let $L$ be a dephased binary log-Hadamard matrix of order $16$, or of order $32$ and rank at most seven, and let $r=\rank L$. Then every canonical column set of $L$ tiles $\F_2^r$.
\end{proposition}

\begin{proof}
Let $m$ be the order of $L$ and $H=((-1)^{L_{ij}})$. If $m=16$, $H$ is equivalent to one of the five representatives, whose canonical column sets tile by Table~\ref{tab:sixteen}; Lemma~\ref{lem:equivalence} transfers the tiling to $L$. If $m=32$, $H$ is equivalent to $R$ or to $R^{\T}$ for the representative $R$ of some extended class. By Lemma~\ref{lem:equivalence} and Remark~\ref{rem:transpose}, $R$ has dephased rank $r\le7$, so $R$ is one of the sixteen records of Table~\ref{tab:witnesses}. The certificates for $R$ and $R^{\T}$ and Lemma~\ref{lem:equivalence} again give the tiling.
\end{proof}

\section{Proof of the main theorem}\label{sec:proof}

\begin{proof}[Proof of Theorem~\ref{thm:main}]
Proposition~\ref{prop:forward} shows that every tile of $\F_2^7$ is spectral. Conversely, let $E\subseteq\F_2^7$ be spectral. By Corollary~\ref{cor:cardinality}, $|E|$ is a power of two. Sizes $1$, $2$ and $4$ follow from Lemma~\ref{lem:small}, size $8$ from Lemma~\ref{lem:eight}, and sizes $64$ and $128$ from Corollary~\ref{cor:cardinality}. Let $|E|\in\{16,32\}$, and let $B$ be a spectrum of $E$. After translating $E$ and $B$ we may assume $0\in E\cap B$. The matrix $L=(b\cdot e)$ of Lemma~\ref{lem:canonical} has rank at most seven, because it is the product of a $|B|\times7$ and a $7\times|E|$ matrix. By Proposition~\ref{prop:smallrank} its canonical column sets tile $\F_2^{\rank L}$, and Lemma~\ref{lem:canonical} shows that $E$ tiles $\F_2^7$.
\end{proof}

\begin{corollary}\label{cor:sixteen}
For every $d$, every sixteen-point spectral subset of $\F_2^d$ tiles $\F_2^d$.
\end{corollary}

\begin{proof}
Apply Lemma~\ref{lem:canonical} and Proposition~\ref{prop:smallrank} to a translated spectral pair.
\end{proof}

\begin{remark}\label{rem:dim8}
In $\F_2^8$ only a few cases remain open after the arguments above. Tiles of size at least $32$ are spectral by Corollary~\ref{cor:large-tiles}, and sets of size at most eight are treated by Lemmas~\ref{lem:small} and~\ref{lem:eight}, so the tile-to-spectral implication there is open only for tiles of size $16$. By Remark~\ref{rem:smith-other} and Corollary~\ref{cor:sixteen}, the converse implication is open only for spectral sets of sizes $32$ and $64$. For size $32$ the method of this paper would require the tiling of the canonical column sets of the $166$ extended classes of order $32$ with dephased rank at most eight (Table~\ref{tab:ranks}), in both orientations; for size $64$ it would require information on Hadamard matrices of order $64$, for which we know of no complete classification.
\end{remark}

\appendix

\section{Data availability}\label{app:data}

The certificates of orders $16$ and $32$, the sixteen retained records of order $32$, and the programs used for the rank scan and the certificate checks are provided as ancillary files with the arXiv version of this paper, together with a README file that identifies the data used and describes how to reproduce the computations. The ancillary files are the version used for this paper; the same material is also available at \href{https://github.com/KnowledgeSense/math/tree/main/preprints/Fugledes-Conjecture-Holds-in-F2-7-October-6-2026}{\texttt{github.com/KnowledgeSense/math}}. The classification file of order $32$ is not redistributed; it is available from McKay's database \cite{McK}.

\section{Certificates for order 32}\label{app:certificates}

For each retained record we list its $32$ hexadecimal words (rows $0$ to $31$, eight per line), followed by certificates for $H$ and for $H^{\T}$, with the conventions of Section~\ref{sec:computation}. The point set $E_C$ is listed in column order, so that its $j$-th entry is the canonical point of column $j$; condition~(3) of Section~\ref{sec:computation} can thus be checked directly against the matrix.

\subsection*{Record 25418, rank 5}
\begin{Verbatim}[fontsize=\scriptsize,xleftmargin=1em]
FFFFFFFF FFFF0000 FF00FF00 FF0000FF F0F0F0F0 F0F00F0F F00FF00F F00F0FF0
CCCCCCCC CCCC3333 CC33CC33 CC3333CC C3C3C3C3 C3C33C3C C33CC33C C33C3CC3
AAAAAAAA AAAA5555 AA55AA55 AA5555AA A5A5A5A5 A5A55A5A A55AA55A A55A5AA5
99999999 99996666 99669966 99666699 96969696 96966969 96699669 96696996
\end{Verbatim}
\noindent{\small $H$:\quad $E_C=(0,\allowbreak 16,\allowbreak 8,\allowbreak 24,\allowbreak 4,\allowbreak 20,\allowbreak 12,\allowbreak 28,\allowbreak 2,\allowbreak 18,\allowbreak 10,\allowbreak 26,\allowbreak 6,\allowbreak 22,\allowbreak 14,\allowbreak 30,\allowbreak 31,\allowbreak 15,\allowbreak 23,\allowbreak 7,\allowbreak 27,\allowbreak 11,\allowbreak 19,\allowbreak 3,\allowbreak 29,\allowbreak 13,\allowbreak 21,\allowbreak 5,\allowbreak 25,\allowbreak 9,\allowbreak 17,\allowbreak 1)$;\quad $T=\{0\}$.}\par
\noindent{\small $H^{\mathsf T}$:\quad $E_C=(0,\allowbreak 16,\allowbreak 8,\allowbreak 24,\allowbreak 4,\allowbreak 20,\allowbreak 12,\allowbreak 28,\allowbreak 2,\allowbreak 18,\allowbreak 10,\allowbreak 26,\allowbreak 6,\allowbreak 22,\allowbreak 14,\allowbreak 30,\allowbreak 31,\allowbreak 15,\allowbreak 23,\allowbreak 7,\allowbreak 27,\allowbreak 11,\allowbreak 19,\allowbreak 3,\allowbreak 29,\allowbreak 13,\allowbreak 21,\allowbreak 5,\allowbreak 25,\allowbreak 9,\allowbreak 17,\allowbreak 1)$;\quad $T=\{0\}$.}\par

\subsection*{Record 25419, rank 6}
\begin{Verbatim}[fontsize=\scriptsize,xleftmargin=1em]
FFFFFFFF FFFF0000 FF00FF00 FF0000FF F0F0F0F0 F0F00F0F F00FF00F F00F0FF0
CCCCCCCC CCCC3333 CC33CC33 CC3333CC C3C3C3C3 C3C33C3C C33CC33C C33C3CC3
AAAAAAAA AAAA5555 AA55AA55 AA5555AA A5A5A5A5 A5A55A5A A55AA55A A55A5AA5
99999996 99996669 99669969 99666696 96969699 96966966 96699666 96696999
\end{Verbatim}
\noindent{\small $H$:\quad $E_C=(0,\allowbreak 16,\allowbreak 8,\allowbreak 24,\allowbreak 36,\allowbreak 52,\allowbreak 44,\allowbreak 60,\allowbreak 34,\allowbreak 50,\allowbreak 42,\allowbreak 58,\allowbreak 38,\allowbreak 54,\allowbreak 46,\allowbreak 62,\allowbreak 31,\allowbreak 15,\allowbreak 23,\allowbreak 7,\allowbreak 27,\allowbreak 11,\allowbreak 19,\allowbreak 3,\allowbreak 29,\allowbreak 13,\allowbreak 21,\allowbreak 5,\allowbreak 25,\allowbreak 9,\allowbreak 17,\allowbreak 1)$;\quad $T=\{0,\allowbreak 56\}$.}\par
\noindent{\small $H^{\mathsf T}$:\quad $E_C=(0,\allowbreak 32,\allowbreak 8,\allowbreak 40,\allowbreak 4,\allowbreak 36,\allowbreak 12,\allowbreak 44,\allowbreak 18,\allowbreak 50,\allowbreak 26,\allowbreak 58,\allowbreak 22,\allowbreak 54,\allowbreak 30,\allowbreak 62,\allowbreak 3,\allowbreak 35,\allowbreak 11,\allowbreak 43,\allowbreak 7,\allowbreak 39,\allowbreak 15,\allowbreak 47,\allowbreak 45,\allowbreak 13,\allowbreak 37,\allowbreak 5,\allowbreak 41,\allowbreak 9,\allowbreak 33,\allowbreak 1)$;\quad $T=\{0,\allowbreak 60\}$.}\par

\subsection*{Record 25420, rank 7}
\begin{Verbatim}[fontsize=\scriptsize,xleftmargin=1em]
FFFFFFFF FFFF0000 FF00FF00 FF0000FF F0F0F0F0 F0F00F0F F00FF00F F00F0FF0
CCCCCCCC CCCC3333 CC33CC33 CC3333CC C3C3C3C3 C3C33C3C C33CC33C C33C3CC3
AAAAAAAA AAAA5555 AA55AA55 AA5555AA A5A5A5A5 A5A55A5A A55AA55A A55A5AA5
99999996 99996669 99669969 99666696 96969666 96966999 96699699 96696966
\end{Verbatim}
\noindent{\small $H$:\quad $E_C=(0,\allowbreak 16,\allowbreak 8,\allowbreak 24,\allowbreak 100,\allowbreak 116,\allowbreak 108,\allowbreak 124,\allowbreak 34,\allowbreak 50,\allowbreak 42,\allowbreak 58,\allowbreak 38,\allowbreak 54,\allowbreak 46,\allowbreak 62,\allowbreak 31,\allowbreak 15,\allowbreak 23,\allowbreak 7,\allowbreak 27,\allowbreak 11,\allowbreak 19,\allowbreak 3,\allowbreak 29,\allowbreak 13,\allowbreak 21,\allowbreak 5,\allowbreak 25,\allowbreak 9,\allowbreak 17,\allowbreak 1)$;\quad $T=\{0,\allowbreak 120,\allowbreak 92,\allowbreak 60\}$.}\par
\noindent{\small $H^{\mathsf T}$:\quad $E_C=(0,\allowbreak 32,\allowbreak 8,\allowbreak 40,\allowbreak 68,\allowbreak 100,\allowbreak 76,\allowbreak 108,\allowbreak 82,\allowbreak 114,\allowbreak 90,\allowbreak 122,\allowbreak 22,\allowbreak 54,\allowbreak 30,\allowbreak 62,\allowbreak 107,\allowbreak 75,\allowbreak 99,\allowbreak 67,\allowbreak 47,\allowbreak 15,\allowbreak 39,\allowbreak 7,\allowbreak 5,\allowbreak 37,\allowbreak 13,\allowbreak 45,\allowbreak 41,\allowbreak 9,\allowbreak 33,\allowbreak 1)$;\quad $T=\{0,\allowbreak 124,\allowbreak 104,\allowbreak 60\}$.}\par

\subsection*{Record 25421, rank 6}
\begin{Verbatim}[fontsize=\scriptsize,xleftmargin=1em]
FFFFFFFF FFFF0000 FF00FF00 FF0000FF F0F0F0F0 F0F00F0F F00FF00F F00F0FF0
CCCCCCCC CCCC3333 CC33CC33 CC3333CC C3C3C3C3 C3C33C3C C33CC33C C33C3CC3
AAAAAAAA AAAA5555 AA55AA55 AA5555AA A5A5A5A5 A5A55A5A A55AA55A A55A5AA5
99999966 99996699 99669999 99666666 96969669 96966996 96699696 96696969
\end{Verbatim}
\noindent{\small $H$:\quad $E_C=(0,\allowbreak 16,\allowbreak 8,\allowbreak 24,\allowbreak 4,\allowbreak 20,\allowbreak 12,\allowbreak 28,\allowbreak 34,\allowbreak 50,\allowbreak 42,\allowbreak 58,\allowbreak 38,\allowbreak 54,\allowbreak 46,\allowbreak 62,\allowbreak 31,\allowbreak 15,\allowbreak 23,\allowbreak 7,\allowbreak 27,\allowbreak 11,\allowbreak 19,\allowbreak 3,\allowbreak 29,\allowbreak 13,\allowbreak 21,\allowbreak 5,\allowbreak 25,\allowbreak 9,\allowbreak 17,\allowbreak 1)$;\quad $T=\{0,\allowbreak 60\}$.}\par
\noindent{\small $H^{\mathsf T}$:\quad $E_C=(0,\allowbreak 16,\allowbreak 8,\allowbreak 24,\allowbreak 4,\allowbreak 20,\allowbreak 12,\allowbreak 28,\allowbreak 34,\allowbreak 50,\allowbreak 42,\allowbreak 58,\allowbreak 38,\allowbreak 54,\allowbreak 46,\allowbreak 62,\allowbreak 7,\allowbreak 23,\allowbreak 15,\allowbreak 31,\allowbreak 3,\allowbreak 19,\allowbreak 11,\allowbreak 27,\allowbreak 29,\allowbreak 13,\allowbreak 21,\allowbreak 5,\allowbreak 25,\allowbreak 9,\allowbreak 17,\allowbreak 1)$;\quad $T=\{0,\allowbreak 60\}$.}\par

\subsection*{Record 25422, rank 7}
\begin{Verbatim}[fontsize=\scriptsize,xleftmargin=1em]
FFFFFFFF FFFF0000 FF00FF00 FF0000FF F0F0F0F0 F0F00F0F F00FF00F F00F0FF0
CCCCCCCC CCCC3333 CC33CC33 CC3333CC C3C3C3C3 C3C33C3C C33CC33C C33C3CC3
AAAAAAAA AAAA5555 AA55AA55 AA5555AA A5A5A5A5 A5A55A5A A55A9999 A55A6666
9996A55A 99965AA5 99699966 99696699 96999696 96996969 96669669 96666996
\end{Verbatim}
\noindent{\small $H$:\quad $E_C=(0,\allowbreak 16,\allowbreak 104,\allowbreak 120,\allowbreak 100,\allowbreak 116,\allowbreak 12,\allowbreak 28,\allowbreak 98,\allowbreak 114,\allowbreak 10,\allowbreak 26,\allowbreak 6,\allowbreak 22,\allowbreak 110,\allowbreak 126,\allowbreak 95,\allowbreak 79,\allowbreak 87,\allowbreak 71,\allowbreak 27,\allowbreak 11,\allowbreak 19,\allowbreak 3,\allowbreak 29,\allowbreak 13,\allowbreak 21,\allowbreak 5,\allowbreak 25,\allowbreak 9,\allowbreak 17,\allowbreak 1)$;\quad $T=\{0,\allowbreak 124,\allowbreak 88,\allowbreak 52\}$.}\par
\noindent{\small $H^{\mathsf T}$:\quad $E_C=(0,\allowbreak 16,\allowbreak 8,\allowbreak 24,\allowbreak 4,\allowbreak 20,\allowbreak 12,\allowbreak 28,\allowbreak 98,\allowbreak 114,\allowbreak 106,\allowbreak 122,\allowbreak 102,\allowbreak 118,\allowbreak 110,\allowbreak 126,\allowbreak 47,\allowbreak 63,\allowbreak 39,\allowbreak 55,\allowbreak 43,\allowbreak 59,\allowbreak 61,\allowbreak 45,\allowbreak 3,\allowbreak 19,\allowbreak 21,\allowbreak 5,\allowbreak 25,\allowbreak 9,\allowbreak 17,\allowbreak 1)$;\quad $T=\{0,\allowbreak 116,\allowbreak 84,\allowbreak 48\}$.}\par

\subsection*{Record 25423, rank 7}
\begin{Verbatim}[fontsize=\scriptsize,xleftmargin=1em]
FFFFFFFF FFFF0000 FF00FF00 FF0000FF F0F0F0F0 F0F00F0F F00FF00F F00F0FF0
CCCCCCCC CCCC3333 CC33CC33 CC3333CC C3C3C3C3 C3C33C3C C33CC33C C33C3CC3
AAAAAAAA AAAA5555 AA55AA55 AA5555AA A5A5A5A5 A5A55A5A A55A9996 A55A6669
9996A55A 99965AA5 99699969 99696696 96999699 96996966 96669666 96666999
\end{Verbatim}
\noindent{\small $H$:\quad $E_C=(0,\allowbreak 16,\allowbreak 104,\allowbreak 120,\allowbreak 68,\allowbreak 84,\allowbreak 44,\allowbreak 60,\allowbreak 66,\allowbreak 82,\allowbreak 42,\allowbreak 58,\allowbreak 38,\allowbreak 54,\allowbreak 78,\allowbreak 94,\allowbreak 95,\allowbreak 79,\allowbreak 87,\allowbreak 71,\allowbreak 27,\allowbreak 11,\allowbreak 19,\allowbreak 3,\allowbreak 29,\allowbreak 13,\allowbreak 21,\allowbreak 5,\allowbreak 25,\allowbreak 9,\allowbreak 17,\allowbreak 1)$;\quad $T=\{0,\allowbreak 124,\allowbreak 88,\allowbreak 52\}$.}\par
\noindent{\small $H^{\mathsf T}$:\quad $E_C=(0,\allowbreak 32,\allowbreak 8,\allowbreak 40,\allowbreak 4,\allowbreak 36,\allowbreak 12,\allowbreak 44,\allowbreak 82,\allowbreak 114,\allowbreak 90,\allowbreak 122,\allowbreak 86,\allowbreak 118,\allowbreak 94,\allowbreak 126,\allowbreak 99,\allowbreak 67,\allowbreak 107,\allowbreak 75,\allowbreak 103,\allowbreak 71,\allowbreak 109,\allowbreak 77,\allowbreak 47,\allowbreak 15,\allowbreak 37,\allowbreak 5,\allowbreak 41,\allowbreak 9,\allowbreak 33,\allowbreak 1)$;\quad $T=\{0,\allowbreak 124,\allowbreak 96,\allowbreak 60\}$.}\par

\subsection*{Record 25428, rank 7}
\begin{Verbatim}[fontsize=\scriptsize,xleftmargin=1em]
FFFFFFFF FFFF0000 FF00FF00 FF0000FF F0F0F0F0 F0F00F0F F00FF00F F00F0FF0
CCCCCCCC CCCC3333 CC33CC33 CC3333CC C3C3C3C3 C3C33C3C C33CC33C C33C3CC3
AAAAAAAA AAAA5555 AA55AA55 AA5555AA A5A5A599 A5A55A66 A55AA566 A55A5A99
99999996 99996669 99669969 99666696 969696A5 9696695A 9669965A 966969A5
\end{Verbatim}
\noindent{\small $H$:\quad $E_C=(0,\allowbreak 16,\allowbreak 40,\allowbreak 56,\allowbreak 100,\allowbreak 116,\allowbreak 76,\allowbreak 92,\allowbreak 66,\allowbreak 82,\allowbreak 74,\allowbreak 90,\allowbreak 70,\allowbreak 86,\allowbreak 78,\allowbreak 94,\allowbreak 31,\allowbreak 15,\allowbreak 23,\allowbreak 7,\allowbreak 27,\allowbreak 11,\allowbreak 19,\allowbreak 3,\allowbreak 29,\allowbreak 13,\allowbreak 21,\allowbreak 5,\allowbreak 25,\allowbreak 9,\allowbreak 17,\allowbreak 1)$;\quad $T=\{0,\allowbreak 124,\allowbreak 88,\allowbreak 52\}$.}\par
\noindent{\small $H^{\mathsf T}$:\quad $E_C=(0,\allowbreak 64,\allowbreak 8,\allowbreak 72,\allowbreak 20,\allowbreak 84,\allowbreak 28,\allowbreak 92,\allowbreak 50,\allowbreak 114,\allowbreak 58,\allowbreak 122,\allowbreak 38,\allowbreak 102,\allowbreak 46,\allowbreak 110,\allowbreak 5,\allowbreak 69,\allowbreak 13,\allowbreak 77,\allowbreak 7,\allowbreak 71,\allowbreak 15,\allowbreak 79,\allowbreak 75,\allowbreak 11,\allowbreak 67,\allowbreak 3,\allowbreak 73,\allowbreak 9,\allowbreak 65,\allowbreak 1)$;\quad $T=\{0,\allowbreak 126,\allowbreak 108,\allowbreak 90\}$.}\par

\subsection*{Record 25432, rank 7}
\begin{Verbatim}[fontsize=\scriptsize,xleftmargin=1em]
FFFFFFFF FFFF0000 FF00FF00 FF0000FF F0F0F0F0 F0F00F0F F00FF00F F00F0FF0
CCCCCCCC CCCC3333 CC33CC33 CC3333CC C3C3C3C3 C3C33C3C C33CC33C C33C3CC3
AAAAAAAA AAAA5555 AA55AA55 AA5555AA A5A5A599 A5A55A66 A55AA566 A55A5A99
9999995A 999966A5 996699A5 9966665A 96969669 96966996 96699696 96696969
\end{Verbatim}
\noindent{\small $H$:\quad $E_C=(0,\allowbreak 16,\allowbreak 104,\allowbreak 120,\allowbreak 100,\allowbreak 116,\allowbreak 12,\allowbreak 28,\allowbreak 66,\allowbreak 82,\allowbreak 74,\allowbreak 90,\allowbreak 70,\allowbreak 86,\allowbreak 78,\allowbreak 94,\allowbreak 31,\allowbreak 15,\allowbreak 23,\allowbreak 7,\allowbreak 27,\allowbreak 11,\allowbreak 19,\allowbreak 3,\allowbreak 29,\allowbreak 13,\allowbreak 21,\allowbreak 5,\allowbreak 25,\allowbreak 9,\allowbreak 17,\allowbreak 1)$;\quad $T=\{0,\allowbreak 124,\allowbreak 92,\allowbreak 60\}$.}\par
\noindent{\small $H^{\mathsf T}$:\quad $E_C=(0,\allowbreak 32,\allowbreak 8,\allowbreak 40,\allowbreak 84,\allowbreak 116,\allowbreak 92,\allowbreak 124,\allowbreak 18,\allowbreak 50,\allowbreak 26,\allowbreak 58,\allowbreak 70,\allowbreak 102,\allowbreak 78,\allowbreak 110,\allowbreak 7,\allowbreak 39,\allowbreak 15,\allowbreak 47,\allowbreak 5,\allowbreak 37,\allowbreak 13,\allowbreak 45,\allowbreak 43,\allowbreak 11,\allowbreak 35,\allowbreak 3,\allowbreak 41,\allowbreak 9,\allowbreak 33,\allowbreak 1)$;\quad $T=\{0,\allowbreak 126,\allowbreak 106,\allowbreak 60\}$.}\par

\subsection*{Record 25434, rank 7}
\begin{Verbatim}[fontsize=\scriptsize,xleftmargin=1em]
FFFFFFFF FFFF0000 FF00FF00 FF0000FF F0F0F0F0 F0F00F0F F00FF00F F00F0FF0
CCCCCCCC CCCC3333 CC33CC33 CC3333CC C3C3C3C3 C3C33C3C C33CC33C C33C3CC3
AAAAAAAA AAAA5555 AA55AA55 AA5555AA A5A5A599 A5A55A66 A55AA566 A55A5A99
999996A5 9999695A 9966965A 996669A5 96969996 96966669 96699969 96696696
\end{Verbatim}
\noindent{\small $H$:\quad $E_C=(0,\allowbreak 16,\allowbreak 104,\allowbreak 120,\allowbreak 100,\allowbreak 116,\allowbreak 12,\allowbreak 28,\allowbreak 66,\allowbreak 82,\allowbreak 74,\allowbreak 90,\allowbreak 6,\allowbreak 22,\allowbreak 14,\allowbreak 30,\allowbreak 31,\allowbreak 15,\allowbreak 23,\allowbreak 7,\allowbreak 27,\allowbreak 11,\allowbreak 19,\allowbreak 3,\allowbreak 29,\allowbreak 13,\allowbreak 21,\allowbreak 5,\allowbreak 25,\allowbreak 9,\allowbreak 17,\allowbreak 1)$;\quad $T=\{0,\allowbreak 124,\allowbreak 88,\allowbreak 52\}$.}\par
\noindent{\small $H^{\mathsf T}$:\quad $E_C=(0,\allowbreak 64,\allowbreak 8,\allowbreak 72,\allowbreak 20,\allowbreak 84,\allowbreak 28,\allowbreak 92,\allowbreak 50,\allowbreak 114,\allowbreak 58,\allowbreak 122,\allowbreak 38,\allowbreak 102,\allowbreak 46,\allowbreak 110,\allowbreak 71,\allowbreak 7,\allowbreak 79,\allowbreak 15,\allowbreak 69,\allowbreak 5,\allowbreak 77,\allowbreak 13,\allowbreak 75,\allowbreak 11,\allowbreak 67,\allowbreak 3,\allowbreak 73,\allowbreak 9,\allowbreak 65,\allowbreak 1)$;\quad $T=\{0,\allowbreak 126,\allowbreak 108,\allowbreak 90\}$.}\par

\subsection*{Record 25438, rank 7}
\begin{Verbatim}[fontsize=\scriptsize,xleftmargin=1em]
FFFFFFFF FFFF0000 FF00FF00 FF0000FF F0F0F0F0 F0F00F0F F00FF00F F00F0FF0
CCCCCCCC CCCC3333 CC33CC33 CC3333CC C3C3C3C3 C3C33C3C C33CC33C C33C3CC3
AAAAAAAA AAAA5555 AA55AA55 AA5555AA A5A5A599 A5A55A66 A55AA566 A55A5A99
99999696 99996969 99669669 99666996 969699A5 9696665A 9669995A 966966A5
\end{Verbatim}
\noindent{\small $H$:\quad $E_C=(0,\allowbreak 16,\allowbreak 40,\allowbreak 56,\allowbreak 100,\allowbreak 116,\allowbreak 76,\allowbreak 92,\allowbreak 2,\allowbreak 18,\allowbreak 10,\allowbreak 26,\allowbreak 70,\allowbreak 86,\allowbreak 78,\allowbreak 94,\allowbreak 31,\allowbreak 15,\allowbreak 23,\allowbreak 7,\allowbreak 27,\allowbreak 11,\allowbreak 19,\allowbreak 3,\allowbreak 29,\allowbreak 13,\allowbreak 21,\allowbreak 5,\allowbreak 25,\allowbreak 9,\allowbreak 17,\allowbreak 1)$;\quad $T=\{0,\allowbreak 124,\allowbreak 90,\allowbreak 54\}$.}\par
\noindent{\small $H^{\mathsf T}$:\quad $E_C=(0,\allowbreak 32,\allowbreak 8,\allowbreak 40,\allowbreak 20,\allowbreak 52,\allowbreak 28,\allowbreak 60,\allowbreak 82,\allowbreak 114,\allowbreak 90,\allowbreak 122,\allowbreak 70,\allowbreak 102,\allowbreak 78,\allowbreak 110,\allowbreak 13,\allowbreak 45,\allowbreak 5,\allowbreak 37,\allowbreak 15,\allowbreak 47,\allowbreak 7,\allowbreak 39,\allowbreak 43,\allowbreak 11,\allowbreak 35,\allowbreak 3,\allowbreak 41,\allowbreak 9,\allowbreak 33,\allowbreak 1)$;\quad $T=\{0,\allowbreak 126,\allowbreak 108,\allowbreak 58\}$.}\par

\subsection*{Record 25467, rank 7}
\begin{Verbatim}[fontsize=\scriptsize,xleftmargin=1em]
FFFFFFFF FFFF0000 FF00FF00 FF0000FF F0F0F0F0 F0F00F0F F00FF00F F00F0FF0
CCCCCCCC CCCC3333 CC33CC33 CC3333CC C3C3C3C3 C3C33C3C C33CC33C C33C3CC3
AAAAAAAA AAAA5555 AA55AA55 AA5555AA A5A5A55A A5A55AA5 A55AA5A5 A55A5A5A
99999696 99996969 99669669 99666996 96969966 96966699 96699999 96696666
\end{Verbatim}
\noindent{\small $H$:\quad $E_C=(0,\allowbreak 16,\allowbreak 8,\allowbreak 24,\allowbreak 68,\allowbreak 84,\allowbreak 76,\allowbreak 92,\allowbreak 98,\allowbreak 114,\allowbreak 106,\allowbreak 122,\allowbreak 38,\allowbreak 54,\allowbreak 46,\allowbreak 62,\allowbreak 31,\allowbreak 15,\allowbreak 23,\allowbreak 7,\allowbreak 27,\allowbreak 11,\allowbreak 19,\allowbreak 3,\allowbreak 29,\allowbreak 13,\allowbreak 21,\allowbreak 5,\allowbreak 25,\allowbreak 9,\allowbreak 17,\allowbreak 1)$;\quad $T=\{0,\allowbreak 126,\allowbreak 94,\allowbreak 56\}$.}\par
\noindent{\small $H^{\mathsf T}$:\quad $E_C=(0,\allowbreak 16,\allowbreak 8,\allowbreak 24,\allowbreak 68,\allowbreak 84,\allowbreak 76,\allowbreak 92,\allowbreak 98,\allowbreak 114,\allowbreak 106,\allowbreak 122,\allowbreak 38,\allowbreak 54,\allowbreak 46,\allowbreak 62,\allowbreak 19,\allowbreak 3,\allowbreak 27,\allowbreak 11,\allowbreak 15,\allowbreak 31,\allowbreak 7,\allowbreak 23,\allowbreak 5,\allowbreak 21,\allowbreak 13,\allowbreak 29,\allowbreak 25,\allowbreak 9,\allowbreak 17,\allowbreak 1)$;\quad $T=\{0,\allowbreak 126,\allowbreak 94,\allowbreak 56\}$.}\par

\subsection*{Record 25472, rank 7}
\begin{Verbatim}[fontsize=\scriptsize,xleftmargin=1em]
FFFFFFFF FFFF0000 FF00FF00 FF0000FF F0F0F0F0 F0F00F0F F00FF00F F00F0FF0
CCCCCCCC CCCC3333 CC33CC33 CC3333CC C3C3C3C3 C3C33C3C C33CC33C C33C3CC3
AAAAAAAA AAAA5555 AA55AA55 AA5555AA A5A59999 A5A56666 A55A9966 A55A6699
99999696 99996969 99669669 99666996 9696A5A5 96965A5A 9669A55A 96695AA5
\end{Verbatim}
\noindent{\small $H$:\quad $E_C=(0,\allowbreak 16,\allowbreak 40,\allowbreak 56,\allowbreak 100,\allowbreak 116,\allowbreak 76,\allowbreak 92,\allowbreak 2,\allowbreak 18,\allowbreak 42,\allowbreak 58,\allowbreak 102,\allowbreak 118,\allowbreak 78,\allowbreak 94,\allowbreak 31,\allowbreak 15,\allowbreak 23,\allowbreak 7,\allowbreak 27,\allowbreak 11,\allowbreak 19,\allowbreak 3,\allowbreak 29,\allowbreak 13,\allowbreak 21,\allowbreak 5,\allowbreak 25,\allowbreak 9,\allowbreak 17,\allowbreak 1)$;\quad $T=\{0,\allowbreak 126,\allowbreak 90,\allowbreak 54\}$.}\par
\noindent{\small $H^{\mathsf T}$:\quad $E_C=(0,\allowbreak 16,\allowbreak 8,\allowbreak 24,\allowbreak 36,\allowbreak 52,\allowbreak 44,\allowbreak 60,\allowbreak 98,\allowbreak 114,\allowbreak 106,\allowbreak 122,\allowbreak 70,\allowbreak 86,\allowbreak 78,\allowbreak 94,\allowbreak 13,\allowbreak 29,\allowbreak 5,\allowbreak 21,\allowbreak 15,\allowbreak 31,\allowbreak 7,\allowbreak 23,\allowbreak 27,\allowbreak 11,\allowbreak 19,\allowbreak 3,\allowbreak 25,\allowbreak 9,\allowbreak 17,\allowbreak 1)$;\quad $T=\{0,\allowbreak 126,\allowbreak 92,\allowbreak 58\}$.}\par

\subsection*{Record 25475, rank 7}
\begin{Verbatim}[fontsize=\scriptsize,xleftmargin=1em]
FFFFFFFF FFFF0000 FF00FF00 FF0000FF F0F0F0F0 F0F00F0F F00FF00F F00F0FF0
CCCCCCCC CCCC3333 CC33CC33 CC3333CC C3C3C3C3 C3C33C3C C33CC33C C33C3CC3
AAAAAAAA AAAA5555 AA55AA55 AA5555AA A5A59999 A5A56666 A55A9966 A55A6699
9996A596 99965A69 9969A569 99695A96 969996A5 9699695A 9666965A 966669A5
\end{Verbatim}
\noindent{\small $H$:\quad $E_C=(0,\allowbreak 16,\allowbreak 40,\allowbreak 56,\allowbreak 100,\allowbreak 116,\allowbreak 76,\allowbreak 92,\allowbreak 66,\allowbreak 82,\allowbreak 42,\allowbreak 58,\allowbreak 38,\allowbreak 54,\allowbreak 78,\allowbreak 94,\allowbreak 95,\allowbreak 79,\allowbreak 87,\allowbreak 71,\allowbreak 27,\allowbreak 11,\allowbreak 19,\allowbreak 3,\allowbreak 29,\allowbreak 13,\allowbreak 21,\allowbreak 5,\allowbreak 25,\allowbreak 9,\allowbreak 17,\allowbreak 1)$;\quad $T=\{0,\allowbreak 126,\allowbreak 80,\allowbreak 62\}$.}\par
\noindent{\small $H^{\mathsf T}$:\quad $E_C=(0,\allowbreak 32,\allowbreak 8,\allowbreak 40,\allowbreak 84,\allowbreak 116,\allowbreak 92,\allowbreak 124,\allowbreak 66,\allowbreak 98,\allowbreak 74,\allowbreak 106,\allowbreak 22,\allowbreak 54,\allowbreak 30,\allowbreak 62,\allowbreak 37,\allowbreak 5,\allowbreak 45,\allowbreak 13,\allowbreak 39,\allowbreak 7,\allowbreak 47,\allowbreak 15,\allowbreak 43,\allowbreak 11,\allowbreak 35,\allowbreak 3,\allowbreak 41,\allowbreak 9,\allowbreak 33,\allowbreak 1)$;\quad $T=\{0,\allowbreak 126,\allowbreak 110,\allowbreak 56\}$.}\par

\subsection*{Record 25485, rank 7}
\begin{Verbatim}[fontsize=\scriptsize,xleftmargin=1em]
FFFFFFFF FFFF0000 FF00FF00 FF0000FF F0F0F0F0 F0F00F0F F00FF00F F00F0FF0
CCCCCCCC CCCC3333 CC33CC33 CC3333CC C3C3C3C3 C3C33C3C C33CC33C C33C3CC3
AAAAAAAA AAAA5555 AA55AA55 AA5555AA A5A59996 A5A56669 A55A9969 A55A6696
999996A5 9999695A 9966965A 996669A5 9696A599 96965A66 9669A566 96695A99
\end{Verbatim}
\noindent{\small $H$:\quad $E_C=(0,\allowbreak 16,\allowbreak 104,\allowbreak 120,\allowbreak 68,\allowbreak 84,\allowbreak 44,\allowbreak 60,\allowbreak 98,\allowbreak 114,\allowbreak 74,\allowbreak 90,\allowbreak 6,\allowbreak 22,\allowbreak 46,\allowbreak 62,\allowbreak 31,\allowbreak 15,\allowbreak 23,\allowbreak 7,\allowbreak 27,\allowbreak 11,\allowbreak 19,\allowbreak 3,\allowbreak 29,\allowbreak 13,\allowbreak 21,\allowbreak 5,\allowbreak 25,\allowbreak 9,\allowbreak 17,\allowbreak 1)$;\quad $T=\{0,\allowbreak 124,\allowbreak 88,\allowbreak 52\}$.}\par
\noindent{\small $H^{\mathsf T}$:\quad $E_C=(0,\allowbreak 64,\allowbreak 8,\allowbreak 72,\allowbreak 20,\allowbreak 84,\allowbreak 28,\allowbreak 92,\allowbreak 54,\allowbreak 118,\allowbreak 62,\allowbreak 126,\allowbreak 34,\allowbreak 98,\allowbreak 42,\allowbreak 106,\allowbreak 15,\allowbreak 79,\allowbreak 7,\allowbreak 71,\allowbreak 69,\allowbreak 5,\allowbreak 77,\allowbreak 13,\allowbreak 3,\allowbreak 67,\allowbreak 11,\allowbreak 75,\allowbreak 73,\allowbreak 9,\allowbreak 65,\allowbreak 1)$;\quad $T=\{0,\allowbreak 124,\allowbreak 110,\allowbreak 90\}$.}\par

\subsection*{Record 25487, rank 7}
\begin{Verbatim}[fontsize=\scriptsize,xleftmargin=1em]
FFFFFFFF FFFF0000 FF00FF00 FF0000FF F0F0F0F0 F0F00F0F F00FF00F F00F0FF0
CCCCCCCC CCCC3333 CC33CC33 CC3333CC C3C3C3C3 C3C33C3C C33CC33C C33C3CC3
AAAAAAAA AAAA5555 AA55AA55 AA5555AA A5A59996 A5A56669 A55A9969 A55A6696
9996A599 99965A66 9969A566 99695A99 969996A5 9699695A 9666965A 966669A5
\end{Verbatim}
\noindent{\small $H$:\quad $E_C=(0,\allowbreak 16,\allowbreak 40,\allowbreak 56,\allowbreak 4,\allowbreak 20,\allowbreak 44,\allowbreak 60,\allowbreak 34,\allowbreak 50,\allowbreak 74,\allowbreak 90,\allowbreak 70,\allowbreak 86,\allowbreak 46,\allowbreak 62,\allowbreak 95,\allowbreak 79,\allowbreak 87,\allowbreak 71,\allowbreak 27,\allowbreak 11,\allowbreak 19,\allowbreak 3,\allowbreak 29,\allowbreak 13,\allowbreak 21,\allowbreak 5,\allowbreak 25,\allowbreak 9,\allowbreak 17,\allowbreak 1)$;\quad $T=\{0,\allowbreak 124,\allowbreak 88,\allowbreak 52\}$.}\par
\noindent{\small $H^{\mathsf T}$:\quad $E_C=(0,\allowbreak 64,\allowbreak 8,\allowbreak 72,\allowbreak 20,\allowbreak 84,\allowbreak 28,\allowbreak 92,\allowbreak 54,\allowbreak 118,\allowbreak 62,\allowbreak 126,\allowbreak 34,\allowbreak 98,\allowbreak 42,\allowbreak 106,\allowbreak 69,\allowbreak 5,\allowbreak 77,\allowbreak 13,\allowbreak 15,\allowbreak 79,\allowbreak 7,\allowbreak 71,\allowbreak 3,\allowbreak 67,\allowbreak 11,\allowbreak 75,\allowbreak 73,\allowbreak 9,\allowbreak 65,\allowbreak 1)$;\quad $T=\{0,\allowbreak 124,\allowbreak 110,\allowbreak 90\}$.}\par

\subsection*{Record 44186, rank 7}
\begin{Verbatim}[fontsize=\scriptsize,xleftmargin=1em]
FFFFFFFF FFFF0000 FF00FF00 FF0000FF F0F0F0F0 F0F00F0F F00FF00F F00F0FF0
CCCCCCCC CCCC3333 CC33CC33 CC3333CC C3C3C3AA C3C33C55 C33CC355 C33C3CAA
AAAAAAC3 AAAA553C A995AA3C A99555C3 A665A5A5 A6655A5A A55AA55A A55A5AA5
9A599999 9A596666 99669966 99666699 96969696 96966969 95A99669 95A96996
\end{Verbatim}
\noindent{\small $H$:\quad $E_C=(0,\allowbreak 16,\allowbreak 56,\allowbreak 40,\allowbreak 52,\allowbreak 36,\allowbreak 12,\allowbreak 28,\allowbreak 114,\allowbreak 82,\allowbreak 122,\allowbreak 90,\allowbreak 118,\allowbreak 86,\allowbreak 126,\allowbreak 94,\allowbreak 111,\allowbreak 79,\allowbreak 103,\allowbreak 71,\allowbreak 107,\allowbreak 75,\allowbreak 35,\allowbreak 3,\allowbreak 109,\allowbreak 77,\allowbreak 37,\allowbreak 5,\allowbreak 41,\allowbreak 9,\allowbreak 33,\allowbreak 1)$;\quad $T=\{0,\allowbreak 96,\allowbreak 120,\allowbreak 24\}$.}\par
\noindent{\small $H^{\mathsf T}$:\quad $E_C=(0,\allowbreak 32,\allowbreak 72,\allowbreak 104,\allowbreak 68,\allowbreak 100,\allowbreak 12,\allowbreak 44,\allowbreak 66,\allowbreak 98,\allowbreak 10,\allowbreak 42,\allowbreak 23,\allowbreak 55,\allowbreak 95,\allowbreak 127,\allowbreak 62,\allowbreak 30,\allowbreak 54,\allowbreak 22,\allowbreak 43,\allowbreak 11,\allowbreak 35,\allowbreak 3,\allowbreak 45,\allowbreak 13,\allowbreak 37,\allowbreak 5,\allowbreak 41,\allowbreak 9,\allowbreak 33,\allowbreak 1)$;\quad $T=\{0,\allowbreak 125,\allowbreak 108,\allowbreak 49\}$.}\par

\section*{Use of AI-assisted tools}

AI-assisted tools were used in the research and preparation of this manuscript.

\end{document}